\documentclass{article}

\usepackage{microtype}
\usepackage{graphicx}
\usepackage{subfigure}
\usepackage{booktabs} % for professional tables
\usepackage{amssymb,amsmath,amsthm}
\usepackage{enumitem}

\usepackage{hyperref}

\usepackage[accepted]{icml2019}

\newtheorem{theorem}{Theorem}[section]

\newtheorem{proposition}[theorem]{Proposition}

\numberwithin{equation}{section}

\DeclareMathOperator\inv{inv}

\DeclareMathOperator\Bin{Bin}
\DeclareMathOperator\argmin{argmin}
\icmltitlerunning{Mallows Ranking Models}

\begin{document}

\twocolumn[
\icmltitle{Mallows Ranking Models: Maximum Likelihood Estimate and Regeneration}



\icmlsetsymbol{equal}{*}

\begin{icmlauthorlist}
\icmlauthor{Wenpin Tang}{UCLA}
\end{icmlauthorlist}

\icmlaffiliation{UCLA}{Department of Mathematics, University of California, Los Angeles, USA}

\icmlcorrespondingauthor{Wenpin Tang}{wenpintang@math.ucla.edu}

\icmlkeywords{Bias, convergence rate, Infinite Generalized Mallows model, large deviations, Mallows model, maximum likelihood estimator, random permutations, ranked data, regeneration}

\vskip 0.3in
]



\printAffiliationsAndNotice{} % otherwise use the standard text.

\begin{abstract}
This paper is concerned with various Mallows ranking models.
We study the statistical properties of the MLE of Mallows' $\phi$ model.
We also make connections of various Mallows ranking models, encompassing recent progress in mathematics. Motivated by the infinite top-$t$ ranking model, we propose an algorithm to select the model size $t$ automatically. 
The key idea relies on the renewal property of such an infinite random permutation.
Our algorithm shows good performance on several data sets.
\end{abstract}

\section{Introduction}
Ranked data appear in many problems of social choice, user recommendation and information retrieval.
Examples include ranking candidates by a large number of voters in an election (e.g. {\em instant-runoff voting}), and the document retrieval problem where one aims to design a meta-search engine according to a ranked list of web pages output by various search algorithms.
In the sequel, we use the words {\em ranking} and {\em permutation} interchangeably.
A ranking model is given by a collection of items, and an unknown total ordering of these items.

There is a rich body of literature on probabilistic ranking models.
The earliest work dates back to \cite{Thur27, Thur31} where items are ranked according to the order statistics of a Gaussian random vector.
\cite{BT52} introduced an exponential family model by pairwise comparisons,
which was extended by \cite{Luce,Plackett} with comparisons of multiple items.
See also \cite{Hunter04, Cat12, CS15, SW17, HF16} for algorithms and statistical analysis of the Bradley-Terry model and its variants.

A more tractable subclass of the Bradley-Terry model was proposed by \cite{Mallows57} as follows.
 For $n \ge 1$, let $\mathfrak{S}_n$ be the set of permutations of $[n]: = \{1, \ldots, n\}$.
The parametric model
 \begin{equation}
 \label{eq:Mallowsperm}
\mathbb{P}_{\theta, \pi_0, d}(\pi) = \frac{1}{\Psi(\theta, d)} e^{- \theta d (\pi ,\pi_0)} \quad \mbox{for } \pi \in \mathfrak{S}_n,
 \end{equation}
is referred to as the {\em Mallows model}. 
Here $\theta > 0$ is the {\em dispersion parameter}, $\pi_0$ is the {\em central ranking}, $d(\cdot, \cdot): \mathfrak{S}_n \times \mathfrak{S}_n \rightarrow \mathbb{R}_{+}$ is a discrepancy function which is right invariant:
$d(\pi, \sigma) = d(\pi \circ \sigma^{-1}, id) \quad \mbox{for } \pi, \sigma \in \mathfrak{S}_n$,
and $\Psi(\theta, d): = \sum_{\pi \in \mathfrak{S}_n} e^{- \theta d (\pi ,\pi_0)}$ is the normalizing constant. 
%For $\theta = 0$, $\mathbb{P}_{0, \pi_0, d}$ is the uniform distribution on $\mathfrak{S}_n$.
Mallows primarily considered two special cases of \eqref{eq:Mallowsperm}:
\begin{itemize}[itemsep = 3 pt]
\item
{\em Mallows' $\theta$ model}, where $d(\pi, \sigma) = \sum_{i = 1}^n (\pi(i) - \sigma(i))^2$ is the {\em Spearman's rho},
\item
{\em Mallows' $\phi$ model}, where $d(\pi, \sigma) = \inv(\pi \circ \sigma^{-1})$, called the {\em Kendall's tau},
\end{itemize}
where $\inv(\pi):= \# \{(i,j) \in [n]^2: i < j \mbox{ and } \pi(i) > \pi(j)\}$ is the number of inversions of $\pi$. 
The general form \eqref{eq:Mallowsperm} was suggested by \cite{Diaconis88} along with other discrepancy functions. 
%as the Hamming distance, the Cayley distance and so on. 
\cite{Diaconis88, Diaconis89} and \cite{Critchlow85} also pioneered the group representation approach to ranked, and partially ranked data.

In this paper we are primarily concerned with the statistical properties of the maximum likelihood estimate (MLE) of the Mallows' $\phi$ model and its infinite counterpart.
Specializing \eqref{eq:Mallowsperm} with the Kendall's tau, the Mallows' $\phi$ model is expressed as
\begin{equation}
\label{eq:Mallowsphi}
\mathbb{P}_{\theta, \pi_0}( \pi) = \frac{1}{\Psi(\theta)} e^{- \theta \inv(\pi \circ \pi_0^{-1})} \quad \mbox{for } \pi \in \mathfrak{S}_n.
\end{equation}
%In \cite{Mallows57}, the model was parameterized as $\mathbb{P}_{\phi, \pi_0}(\pi) = \Psi(- \ln \phi)^{-1} \phi^{\inv(\pi \circ \pi_0^{-1})}$ with $\phi = e^{-\theta}$.
It is easily seen that $\mathbb{P}_{\theta, \pi_0}$ has a unique mode $\pi_0$ if $\theta > 0$.
The Mallows' $\phi$ model \eqref{eq:Mallowsphi} is of particular interest, since it is an instance of two large classes of ranking models: 
distance-based ranking models \cite{FV86} and multistage ranking models \cite{FV88}.
The one parameter model \eqref{eq:Mallowsphi} also has an $n-1$ parameter extension where $\theta$ is replaced with $\vec{\theta}:= (\theta_1, \ldots \theta_{n-1})$ by factorizing the inversion.
This $n-1$ parameter model, called the {\em Generalized Mallows} (GM$_{\vec{\theta}, \pi_0}$) model, will be discussed in Section \ref{s2}.
See \cite{CFV91, Marden14} for a review of these ranking models.

\cite{FV88} showed that if the central ranking $\pi_0$ is known, the MLE of $\theta$ (or $\vec{\theta}$) can be easily found by convex optimization. 
But it is harder to compute the MLE of $\pi_0$, and only a few heuristic algorithms are available. 
As pointed out in \cite{MPPB12}, the problem of finding the MLE of $\pi_0$ for the Mallows' $\phi$ model is the {\em Kemeny's consensus ranking problem} which is known to be NP-hard \cite{Young86, BTT89}.
They also gave a branch and bound search algorithm to estimate simultaneously $\theta$ (or $\vec{\theta}$) and $\pi_0$.
See also \cite{CSS, ACN, MM09} for approximation algorithms for consensus ranking problem,
\cite{BHS14, ICL16, ICL18} for efficient sampling and learning of Mallows models, 
and \cite{LM08, CDHKL} for various generalizations of Mallows ranking models.
Though there have been efforts in developing algorithms to estimate $\theta$ (or $\vec{\theta}$) and $\pi_0$, not much is known about the statistical properties of the MLE $\widehat{\theta}$ and $\widehat{\pi}_0$ even for the simplest model \eqref{eq:Mallowsphi}.
We provide statistical analysis to the MLE of the model \eqref{eq:Mallowsphi}, and answer the following questions in Section \ref{stheory}:
Are the MLEs $\widehat{\theta}$, $\widehat{\pi}_0$ consistent ?
Is the MLE $\widehat{\theta}$ unbiased ?
How fast does the MLE $\widehat{\pi}_0$ converge to $\pi_0$ ?

When the number of items $n$ is large, learning a complete ranking model becomes impracticable. 
A line of work by \cite{FV86, BOB, MB10, MC12} focused on the top-$t$ orderings for the GM$_{\vec{\theta}, \pi_0}$ model.
Among these work, \cite{MB10} proposed a probability model over the top-$t$ orderings of infinite permutations, called the {\em Infinite Generalized Mallows} (IGM$_{\vec{\theta}, \pi_0}$) model:
\begin{equation}
\label{eq:IGM}
\mathbb{P}_{\vec{\theta}, \pi_0}(\pi) = \frac{1}{\Psi(\vec{\theta})} \cdot \exp \left(-\sum_{j = 1}^t \theta_j s_j (\pi \circ \pi_0^{-1})\right),
\end{equation}
where $s(\pi) := (s_1(\pi), s_2(\pi), \ldots)$ is the {\em inversion table} of $\pi$ defined by 
\begin{equation}
\label{eq:sj}
s_j(\pi):= \pi^{-1}(j) - 1 - \sum_{j' < j} 1_{\{\pi^{-1}(j') < \pi^{-1}(j)\}},
\end{equation}
and $\Psi(\vec{\theta}) = \prod_{j = 1}^{t} (1-e^{-\theta_j})^{-1}$ is the normalizing constant.
By convention, $\theta_j = 0$ for any $j > t$.
The integer `$t$' is referred to as the {\em model size} of the IGM model.
%The model \eqref{eq:IGM} is useful to tackle the problem of ranking a large number items (e.g. retrieving information of a target from web pages output by various search engines).
If $\theta_ 1 = \cdots = \theta_t = \theta$, the IGM$_{\vec{\theta}, \pi_0}$ model is called the {\em single parameter} IGM model.

As explained in \cite{MB10}, the IGM$_{\vec{\theta}, \pi_0}$ model \eqref{eq:IGM} is the marginal distribution of a random permutation of positive integers. 
The single parameter IGM was called the {\em infinite $q$-shuffle} in \cite{GO09} with parameterization $q = e^{-\theta}$. 
They provided a nice construction of the infinite $q$-shuffle, which is reminiscent of {\em absorption sampling} \cite{Rawlings, Kemp} and the {\em repeated insertion model} \cite{DPR}.
The infinite $q$-shuffle was further extended by \cite{PT17} to the {\em $p$-shifted permutations} of positive integers.
But the link between Meil\v{a}-Bao's infinite ranking model and infinite $q$-shuffle or $p$-shifted permutations does not seem to have been previously noticed. 
So we point out this connection which will be detailed in Section \ref{s2}.

One disadvantage of the aforementioned top-$t$ ranking models is that they all require choosing `$t$' manually.
Small model size `$t$' often leads to poor accuracy, and large model size `$t$' needs a considerable amount of training time.
It was observed in \cite{PT17} that the random infinite limit of the single parameter IGM model has a remarkable renewal property. 
This suggests a heuristic procedure to select the model size `$t$' automatically based on Meil\v{a}-Bao's search algorithms.
We will discuss such an approach in Section \ref{s2}.

To summarize, the main contributions of this paper are:
\begin{itemize}
\item
Provide statistical analysis to the MLE of the Mallows' $\phi$ model as well as the single parameter IGM.
\item
Propose a selection algorithm for the model size `$t$' of the top-$t$ Mallows ranking models. 
\end{itemize}
See also \cite{LL02, LB11, LB14, ABSV14, VSCFA, LM18} for the approximate Baysian inference of Mallows' mixture models for clustering heterogeneous ranked data,
and \cite{HG12, HKG, MM14} for hierarchical ranking models.

\section{Mallows Models: `$t$' Selection Algorithm}
\label{s2}
In this section we provide background on various Mallows ranking models, encompassing the closely related $q$-shuffles and $p$-shifted permutations.
We also give an algorithm to select the model size `$t$' for the IGM model \eqref{eq:IGM}.
We follow closely \cite{FV88, MB10, PT17}.

\subsection{Finite Mallows Models}
Given $n$ items labelled by $[n]$, a ranking $\pi \in \mathfrak{S}_n$ is represented by
\begin{itemize}[itemsep = 3 pt]
\item
the {\em word list} $(\pi(1), \pi(2), \ldots, \pi(n))$,
\item
the {\em ranked list} $(\pi^{-1}(1)| \pi^{-1}(2) |\ldots |\pi^{-1}(n))$.
\end{itemize}
Here $\pi(i) = j$ means that the item $i$ has rank $j$, and conversely $\pi^{-1}(j) = i$ means that the $j^{th}$ most preferred is item $i$.
The idea of multistage ranking is to decompose the ranking procedure into independent stages. 
The most preferred item is selected at the first stage, the best of the remaining at the second stage and so on until the least preferred item is selected. 
The correctness of the choice at any stage is accessed through a central ranking $\pi_0$.
For example, if $\pi_0 = (3|1|2)$, then the ranking $\pi = (3|2|1)$ gives a correct choice at the first stage, since item $3$ is the most preferred in both $\pi$ and $\pi_0$.
But at the second stage, among the two remaining items $1$ and $2$, item $2$ is selected by $\pi$ while the right choice is item $1$ according to $\pi_0$.

For any ranking $\pi \in \mathfrak{S}_n$ and $j \in [n-1]$, let $(s_1(\pi), \ldots, s_{n-1}(\pi))$ be the inversion table of $\pi$ defined by \eqref{eq:sj}.
It is easy to see that $s_j(\pi) \in \{0, \ldots, n -j\}$, and there is a bijection between a ranking $\pi$ and the inversion table $(s_1(\pi), \ldots, s_{n-1}(\pi))$.
The quantity $s_j(\pi \circ \pi_0^{-1})$ measures the correctness of the choice at stage $j$: 
$s_j(\pi \circ \pi_0^{-1}) = k$ means that at stage $j$ the $(k+1)^{th}$ best of the remaining items is selected. 
In the example with $\pi_0 = (3|1|2)$ and $\pi = (3|2|1)$, we have $s_1(\pi \circ \pi_0^{-1}) = 0$ and $s_2(\pi \circ \pi_0^{-1}) = 1$.
\cite{FV86} introduced the multistage ranking models of the form:
\begin{equation}
\label{eq:FVmultistage}
\mathbb{P}_{p, \pi_0}(\pi) = \prod_{j = 1}^{n-1} p_{j}\left(s_j(\pi \circ \pi_0^{-1})\right), 
\end{equation}
where $p_{j}(\cdot)$ is a probability distribution on $\{0, \ldots, n - j\}$ at stage $j$. 
The choice of $p_j(k) = (1 - e^{-\theta}) e^{-k \theta }/(1 - e^{-(n-j+1)\theta})$ for $k \in \{0, \ldots, n-j\}$ in \eqref{eq:FVmultistage}, and the remarkable identity $\sum_{j = 1}^{n-1} s_j(\pi) = \inv(\pi)$ for any $\pi \in \mathfrak{S}_n$ yield the Mallows' $\phi$ model \eqref{eq:Mallowsphi}.

This one parameter model has a natural $n-1$ parameter extension by simply taking 
$p_j(k) = (1 - e^{-\theta_j})e^{- \theta_j k}/(1 - e^{-(n-j+1) \theta_j})$  for $k \in \{0, \ldots, n-j\}$.
The $n - 1$ parameter model, called the {\em Generalized Mallows} (GM$_{\vec{\theta}, \pi_0}$) model, is then defined by
\begin{equation}
\label{eq:GMM}
\mathbb{P}_{\vec{\theta}, \pi_0}(\pi) = \frac{1}{\Psi(\vec{\theta})}  \exp \left( - \sum_{j = 1}^{n-1}\theta_j s_j(\pi \circ \pi_0^{-1}) \right),
\end{equation}
where $\Psi(\vec{\theta}) = \prod_{j = 1}^{n-1} (1 - e^{-(n-j+1) \theta_j})(1 - e^{-\theta_j})^{-1}$ is the normalizing constant.
The GM$_{\vec{\theta}, \pi_0}$ model is also called the {\em Mallows' $\phi$-component} model.
The model \eqref{eq:GMM} can also be expressed in the form
$e^{-d_{\vec{\theta}}}(\pi, \pi_0)/\Psi(\vec{\theta})$,
with $d_{\vec{\theta}}(\pi, \pi_0): = \sum_{j = 1}^{n-1} \theta_j s_j(\pi \circ \pi_0^{-1})$.
% Contrary to the Mallows' $\phi$ model, $d_{\vec{\theta}}$ is not a distance since it does not satisfy the triangle inequality.

\subsection{Infinite Mallows Models}
Given a countably infinite items labelled by $\mathbb{N}_{+}: = \{1, 2, \ldots \}$, a ranking $\pi$ over $\mathbb{N}_{+}$ is a bijection from $\mathbb{N}_{+}$ onto itself represented by the word list $(\pi(1), \pi(2), \ldots)$ or the ranked list $(\pi^{-1}(1)| \pi^{-1}(2) | \ldots)$.
A top-$t$ ordering of $\pi$ is the prefix $(\pi^{-1}(1)| \ldots |\pi^{-1}(t))$.
Motivated by the GM$_{\vec{\theta}, \pi_0}$ model \eqref{eq:GMM},
\cite{MB10} proposed the {\em Infinite Generalized Mallow} (IGM$_{\vec{\theta}, \pi_0}$) model \eqref{eq:IGM}, which can also be put in the form
$e^{-d_{\vec{\theta}}(\pi, \pi_0)}/\Psi(\vec{\theta})$,
with $d_{\vec{\theta}}(\pi, \pi_0): = \sum_{j = 1}^t \theta_j s_j(\pi \circ \pi_0^{-1})$.
In particular, $s_j$ is distributed as Geo$(1- e^{-\theta_j})$ on $\{0,1,\ldots\}$.
As explained in \cite{MB10}, one can regard $\pi$ as a top-$t$ ordering, and $\pi_0$ as an ordering over $\mathbb{N}_{+}$.
If $\theta_1 = \cdots =\theta_t = \theta$, then the model \eqref{eq:IGM} simplifies to 
\begin{equation}
\label{eq:singleIGM}
\mathbb{P}_{\theta, \pi_0}(\pi) = \frac{1}{\Psi(\theta)} \exp \left(-\theta \sum_{j = 1}^t s_j(\pi \circ \pi_0^{-1}) \right),
\end{equation}
called the {\em single parameter} IGM model. 

It is easily seen that the single parameter IGM model \eqref{eq:singleIGM} is the marginal distribution of a random permutation of positive integers. Formally, this random infinite permutation is distributed as
\begin{equation}
\label{eq:GOIGM}
\mathbb{P}_{\theta, \pi_0}(\pi) = \frac{1}{\Psi(\theta)}  \exp \left( - \theta \sum_{j = 1}^{\infty} s_j(\pi \circ \pi_0^{-1}) \right).
\end{equation}
In the terminology of \cite{GO09}, for $\pi$ defined by \eqref{eq:GOIGM}, $\pi \circ \pi_0^{-1}$ is the {\em infinite $e^{-\theta}$-shuffle}.
The latter was generalized by \cite{PT17} to {\em $p$-shifted permutations}, 
where $p = (p_1, p_2, \ldots)$ is a discrete distribution on $\mathbb{N}_{+}$ with $p_1 > 0$.
Here we present a further extension of $p$-shifted permutations.

Let $P = (p_{ij})_{i,j \in \mathbb{N}_{+}}$ be a stochastic matrix on $\mathbb{N}_{+}$ with $p^i = (p_{ij})_{j \in \mathbb{N}_{+}}$ being the $i^{th}$ row of $P$. 
Assume that $\lim_{n \rightarrow \infty} \prod_{i = 1}^n \left(1 - p_{i1} \right) = 0$.
Call a random permutation $\Pi$ of $\mathbb{N}_{+}$ a {\em $P$-shifted permutation} of $\mathbb{N}_{+}$
if $\Pi$ has the distribution defined by the following construction from the independent sample $(X_i)_{i \ge 1}$, with $X_i$ distributed as $p^i$.
Inductively, let $\Pi_1 := X_1$, and for $i \ge 2$, let $\Pi_i := \psi(X_i)$ where $\psi$ is the increasing bijection from $\mathbb{N}_{+}$ to $\mathbb{N}_{+} \setminus \{ \Pi_1, \cdots, \Pi_{i-1} \}$.
For example, if $X_1 = 2$, $X_2 = 1$, $X_3 = 2$, $X_4 = 3$, $X_5 = 4$, $X_6 = 1 \ldots$, then the associated permutation is $(2,1,4,6,8,3,\ldots)$.

%The condition \eqref{eq:gua} guarantees that $\Pi$ such constructed is almost surely a permutation of $\mathbb{N}_{+}.
Now the aforementioned infinite ranking models are subcases of  the $P$-shifted permutations.
\begin{itemize}[itemsep = 3 pt]
\item
If $p^i = p$ with $p_1 > 0$ for all $i$, then we get the $p$-shifted permutation.
\item
If $p^i = \mbox{Geo}(1 - e^{-\theta})$ on $\mathbb{N}_{+}$ for all $i$, then we get the single parameter IGM model, or infinite $e^{-\theta}$-shuffle, 
i.e. $\Pi \stackrel{d}{=} \pi \circ \pi_0^{-1}$ for $\pi$ distributed according to \eqref{eq:GOIGM}.
\item
If $p^i = \mbox{Geo}(1 - e^{-\theta_i})$ on $\mathbb{N}_{+}$ for each $i$, then $\Pi \stackrel{d}{=} \pi \circ \pi_0^{-1}$ for $\pi$ an infinite version of the IGM model \eqref{eq:IGM}.
\end{itemize}

%\begin{figure}[ht]
%\vskip 0.2in
%\begin{center}
%\centerline{\includegraphics[width=\columnwidth]{Diagram.png}}
%\caption{Diagram of Mallows' type ranking models.}
%\label{icml-historical}
%\end{center}
%\vskip -0.2in
%\end{figure}

\subsection{`$t$' Selection Algorithm}
We present an algorithm to select the model size `$t$' automatically for the top-$t$ IGM models.
The heuristic comes from the renewal property of the single parameter IGM model \eqref{eq:singleIGM}.
We need the following vocabulary.

Let $\Pi$ be a permutation of $\mathbb{N}_{+}$.
Call $n \in \mathbb{N}_{+}$ a {\em splitting time} of $\Pi$
if $\Pi$ maps $[1,n]$ onto itself.
The set of splitting times of $\Pi$ is the collection of finite right endpoints of some finite or infinite family of {\em components} of $\Pi$, say $\{I_j\}$. 
So $\Pi$ does not act as a permutation on any proper subinterval of $I_j$.
%$\Pi$ acts on each of its components $I_j$ as an {\em indecomposable permutation} of $I_j$, meaning that 
These components $I_j$ form a partition of $\mathbb{N}_{+}$, which is coarser than the partition by cycles of $\Pi$.
For example, the permutation $\pi = (1)(2,4)(3) \in \mathfrak{S}_4$ induces the partition by components $[1] [2,3,4]$.

The idea is to use the single parameter IGM model to preselect `$t$', which hinges on the renewal property of the latter.
Then we proceed to train a top-$t$ ranking model.
\cite{PT17} proved that for $p = (p_1, p_2, \ldots)$ a discrete distribution with $p_1 > 0$ and $\sum_{i \ge 1} ip_i < \infty$, a $p$-shifted permutation $\Pi$ is a concatenation of independent and identically distributed (i.i.d.) components. That is, $\Pi$ is characterized by $L$ a distribution on $\mathbb{N}_{+}$, and $(Q_n)_{n \ge 1}$ a sequence of distributions on indecomposable permutations such that
\begin{itemize}[itemsep = 3 pt]
\item
the lengths of components $(L_i)_{i \ge 1}$ are i.i.d. as $L$,
\item
given the length of a component $L_i = n_i$, the reduced component defined via conjugation of $\Pi$ by the shift from the component to $[n_i]$ is distributed as $Q_{n_i}$.
\end{itemize}
To illustrate,
$
(\, \underbrace{2, \, 3, \, 4,\, 1}_{L_1 = 4}, \, \underbrace{6,\, 8, \, 7, \,10, \,5, \, 9}_{L_2 = 6}, \, \underbrace{12, \, 13, \, 11}_{L_3 = 3}, \ldots)
$
Moreover, the probability generating function of $L$ is given by
$F(z) = 1 - \frac{1}{1 + \sum_{n = 1}^{\infty} u_n z^n}$,
with $u_n: = \prod_{i = 1}^n \sum_{j = 1}^i p_i$.
Specializing this renewal construction to the single parameter IGM model, 
we get the following proposition which is the foundation of Algorithm \ref{algo1} described right after.

\begin{proposition}
\label{eq:regenMallowsphi}
Let $\Pi$ be a random permutation of $\mathbb{N}_{+}$ distributed as $\mathbb{P}_{\theta, id}$ defined by \eqref{eq:GOIGM}.
Let $L$ be the common distribution of lengths of components of $\Pi$. Then
\begin{equation}
\label{eq:mean}
\mathbb{E}L = \frac{1}{(e^{-\theta}; e^{-\theta})_{\infty}},
\end{equation}
where $(a; q)_{\infty}: = \prod_{k = 0}^{\infty} (1 - a q^{k})$ is the Q-Pochhammer function.
\end{proposition}

\begin{figure}[ht]
\vskip 0.2in
\begin{center}
\centerline{\includegraphics[width=0.75\columnwidth]{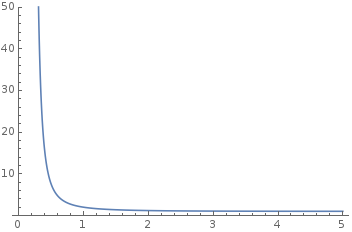}}
\caption{Plot of $\theta \rightarrow 1/(e^{-\theta}; e^{-\theta})_{\infty}$ for $\theta \in  [0,5]$.}
%\label{icml-historical}
\end{center}
\vskip -0.2in
\end{figure}

Pitman-Tang's theory indicates that for the single parameter model \eqref{eq:GOIGM},
the first component of $\pi \circ \pi_0^{-1}$ has length $L$ whose expectation is $1/(e^{-\theta}; e^{-\theta})_{\infty}$.
Given the dispersion parameter $\theta$, a complete permutation is expected to occur at some place close to $1/(e^{-\theta}; e^{-\theta})_{\infty}$.
The latter can be regarded as the {\em effective length} of the random infinite permutation, which suggests a natural candidate for `$t$' in top-$t$ ranking models.
Since $\theta$ is unknown, we would like to find the `$t$' closest to the effective length. 
Given $t$ in a suitable range $\mathbb{T}$, we fit the single parameter IGM model \eqref{eq:singleIGM} to get the MLE $\widehat{\theta}(t)$ by the algorithms in \cite{MB10}.
Those algorithms also work for partially ranked data.
Then we search for 
\begin{equation}
\label{eq:t}
t : = \min \left\{ \argmin_{s \in \mathbb{T}} \left|s - \frac{1}{(e^{-\widehat{\theta}(s)}; e^{-\widehat{\theta}(s)})_{\infty}}\right|, \lambda t_{\max} \right\},
\end{equation}
where $\mathbb{T}$ is the range of search for model sizes, 
$\lambda \in (0,1)$ is a user-defined cutoff fraction to avoid overfitting, and $t_{\max}$ is the maximum length of permutations in the data.
Practically, one starts with `$s$' of small values to narrow down the choices for the effective length.
That way, one only needs to search for a small proportion, not the full range of $[t_{\min}, t_{\max}]$.
With `$t$' selected according to \eqref{eq:t}, we can then fit the IGM model \eqref{eq:IGM} by means of MLE.

\begin{algorithm}[tb]
   \caption{`$t$' selection algorithm}
   \label{algo1}
\begin{algorithmic}
   \STATE $\theta_0 \gets \mbox{MB}(1)$ \qquad \qquad ~ (Run Meil\v{a}-Bao's algorithm)
   
   \STATE Choose $\mathbb{T} \ni 1/(e^{-\theta_0}; e^{-\theta_0})_{\infty}$ of a small range
   
   \STATE Initialize $Err \gets \infty, \,\,  t\_\mbox{SEL} \gets 0$
   
   \FOR{$t$ in $\mathbb{T}$}
   \STATE $\theta \gets \mbox{MB}(t)$ \qquad \qquad (Run Meil\v{a}-Bao's algorithm)
   
   \IF{$|t - 1/(e^{-\theta}; e^{-\theta})_{\infty}| < Err$}
   \STATE $Err \gets |t - 1/(e^{-\theta}; e^{-\theta})_{\infty}|$
   \STATE $ t\_\mbox{SEL} \gets t$
   \ENDIF
   \ENDFOR
   \STATE \textbf{return} $\min(t\_\mbox{SEL}, \lambda t_{\max})$
\end{algorithmic}
\end{algorithm}

\section{Statistical Properties of the MLE}
\label{stheory}
In this section we provide statistical analysis to the MLE of Mallows models. 

\subsection{Main Theorems}
\cite{Mu16} proved that the MLE $\widehat{\theta}$ for the general Mallows model \eqref{eq:Mallowsperm} is consistent if $\pi_0$ is known. 
His approach relies on the concept of {\em permutons} \cite{HKMRS}.
The following result shows that the MLE $\widehat{\theta}$ is always biased upwards for the Mallows' $\phi$ model.
\begin{theorem}[Bias of $\widehat{\theta}$]
\label{thm:bias}
Let $\mathbb{P}_{\theta, \pi_0}$ be defined by \eqref{eq:Mallowsphi}, and $\widehat{\theta}$ be the MLE of $\theta$ with $N$ samples. 
Then for each $N \ge 1$,
\begin{equation}
\label{eq:MIGMbias}
\mathbb{E}_{\theta, \pi_0} \widehat{\theta} > \theta. 
\end{equation}
%Moreover, the error $\widehat{\theta} - \theta$ is of order $1/\sqrt{N}$ with high probability.
\end{theorem}

The analysis of the MLE $\widehat{\pi}_0$ is more subtle since it lives in a discrete space.
By general results of \cite{NM94, CS12}, one can prove the consistency of $\widehat{\pi}_0$. 
Here we establish a concentration bound of $\widehat{\pi}_0$ at $\pi_0$ from which the consistency is straightforward.
This concentration bound also gives a confidence interval for the central ranking in the Mallows' $\phi$ model.

\begin{theorem}[Convergence rate of $\widehat{\pi}_0$]
\label{thm:cvrate}
Let $\mathbb{P}_{\theta, \pi_0}$ be defined by \eqref{eq:Mallowsphi}, and $\widehat{\pi}_0$ be the MLE of $\pi_0$ with $N$ samples.
Then for $N$ large enough,
\begin{equation} 
\mathbb{P}_{\theta, \pi_0}(\widehat{\pi}_0 \ne \pi_0) \ge \frac{1}{1-e^{-\theta}} \sqrt{\frac{2}{\pi N}} \left(\cosh \frac{\theta}{2} \right)^{-N},
\end{equation}
and
\begin{equation} 
\label{eq:upperbd}
\mathbb{P}_{\theta, \pi_0}(\widehat{\pi}_0 \ne \pi_0) \le ( n - H_n) n! \left(\cosh \frac{\theta}{2} \right)^{-N},
\end{equation}
where $H_n : = \sum_{i = 1}^n \frac{1}{i}$ is the harmonic sum.
\end{theorem}

To the best of our knowledge, Theorems \ref{thm:bias} and \ref{thm:cvrate} are new. 
Their proof will be given in the next two subsections.

Similar to Theorem \ref{thm:bias}, the MLE of $\theta$ is biased upwards for the single parameter IGM model, confirming an observation in \cite{MB10}.
\begin{theorem}[Bias of $\widehat{\theta}$]
\label{thm:bias2}
Let $\mathbb{P}_{\theta, \pi_0}$ be the distribution of the single parameter IGM model, and $\widehat{\theta}$ be the MLE of $\theta$ with $N$ samples. Then for each $N \ge 1$,
\begin{equation}
\mathbb{E}_{\theta, \pi_0} \widehat{\theta} > \theta.
\end{equation}
\end{theorem}

Let us mention a few open problems.
The rate of $\widehat{\theta}$ is open for both the Mallows' $\phi$ model and the single parameter IGM model. 
We believe that the rate is of order $1/\sqrt{N}$ for the Mallows' $\phi$ model by applying the delta method with a finer variance analysis.
We can ask the same questions for the generalized Mallows model (GM$_{\vec{\theta}, \pi_0}$).
It is expected that the rate of $\widehat{\pi}$ is of order $e^{-N \beta(\vec{\theta})}$ but the explicit formula of $\beta(\vec{\theta})$ is still missing. 
Deriving a bound of $\beta(\vec{\theta})$ would also be interesting.
The convergence rate of $\widehat{\pi}_0$ for the single parameter IGM model seems to be difficult, since the size of permutations goes to infinity.
We leave the analog of Theorem \ref{thm:cvrate} open.

\subsection{Proof of Theorem \ref{thm:bias}}
We first consider the case where the central ranking $\pi_0$ is known.
Assume w.l.o.g. that $\pi_0 = id$ by suitably relabelling the items. 
Then the model \eqref{eq:Mallowsphi} simplifies to 
\begin{equation}
\label{eq:MallowsID}
\mathbb{P}_{\theta,id}(\pi) = \exp \left(-\theta \inv(\pi) - \ln \Psi(\theta)\right).
\end{equation}
Given $N$ samples $(\pi_i)_{1 \le i \le N}$, the MLE $\widehat{\theta}$ is the solution to the following equation:
\begin{equation}
\label{eq:MLEtheta}
-\frac{\Psi'(\theta)}{\Psi(\theta)} = \frac{1}{N} \sum_{i=1}^N \inv(\pi_i).
\end{equation}
Now by writing $\Psi(\theta) = f(e^{-\theta})$ with $f(q):=\sum_{\pi \in \mathfrak{S}_n} q^{\inv(\pi)}$, we have
\begin{equation*}
-\frac{\Psi'(\theta)}{\Psi(\theta)} = g(e^{-\theta}), \quad \mbox{with} \quad g(q):=\frac{q f'(q)}{f(q)}.
\end{equation*}
So $\widehat{\theta} = - \log g^{-1}(\frac{1}{N} \sum_{i=1}^N \inv(\pi_i))$.
The function $f(q)$ is known as the {\em $q$-factorial} \cite{Stanley}.
We have
\begin{equation}
\label{eq:g}
g(q) = q \sum_{k = 1}^{n-1} \frac{1-(k+1)q^k + kq^{k+1}}{(1-q)(1-q^{k+1})}.
\end{equation}
As observed by \cite{Mallows57, GO09}, for $\pi$ distributed according to \eqref{eq:MallowsID}, the number of inversions has the same distribution as a sum of independent {\em truncated geometric random variables}. That is,
\begin{equation*}
\inv(\pi) \stackrel{d}{=} G_{e^{-\theta},1} + \cdots + G_{e^{-\theta},n} \quad \mbox{for } \pi \sim \mathbb{P}_{\theta, id},
\end{equation*}
where $\mathbb{P}(G_{q,k} = i) = q^i(1-q)/(1-q^k)$ for $i \in \{0,1,\ldots, k-1\}$.
Thus, $\mathbb{E}\inv(\pi) = g(e^{-\theta})$. 
By elementary analysis, $\theta \mapsto g(e^{-\theta})$ is strictly convex and decreasing. 
So its inverse function $q \mapsto -\log g^{-1}(q)$ is strictly convex. 
By Jensen's inequality, 
$- \mathbb{E} \log g^{-1}\left( \frac{1}{N} \sum_{i=1}^N \inv(\pi_i) \right) 
 > - \log g^{-1}\left( \mathbb{E}\left( \frac{1}{N} \sum_{i=1}^N \inv(\pi_i) \right) \right)$,
which implies that $\mathbb{E} \widehat{\theta} > \theta$.

Now consider the case where the central ranking $\pi_0$ is unknown. 
The MLE $\widehat{\theta}$ is given by
\begin{equation}
\widehat{\theta} = - \log g^{-1} \left(\frac{1}{N} \sum_{i = 1}^N \inv(\pi_i \circ \widehat{\pi}_0^{-1}) \right),
\end{equation}
where $g$ is defined as in \eqref{eq:g}, and $\widehat{\pi}_0$ is the MLE of $\pi_0$.
By the definition of $\widehat{\pi}_0$,
\begin{equation*}
\sum_{i = 1}^N \inv(\pi_i \circ \widehat{\pi}_0^{-1}) \le \sum_{i = 1}^N \inv(\pi_i \circ \pi_0^{-1}).
\end{equation*}
Since $q \mapsto -\log g^{-1}(q)$ is strictly convex and decreasing, we get
\begin{align*}
\mathbb{E}\widehat{\theta} & = - \mathbb{E} \log g^{-1}\left(\frac{1}{N}   \sum_{i = 1}^N \inv(\pi_i \circ \widehat{\pi}_0^{-1}) \right) \\
&\ge - \mathbb{E} \log g^{-1}\left(\frac{1}{N}   \sum_{i = 1}^N \inv(\pi_i \circ \pi_0^{-1}) \right) \\
& > - \log g^{-1}\left( \mathbb{E}\left( \frac{1}{N} \sum_{i=1}^N \inv(\pi_i \circ \pi_0^{-1}) \right) \right) = \theta,
\end{align*}
where the last equality follows from the fact that $\pi_i \circ \pi_0^{-1}$ is distributed according to \eqref{eq:MallowsID} for each $i$.

\subsection{Proof of Theorem \ref{thm:cvrate}}
Assume w.l.o.g. that the true central ranking $\pi_0 = id$. 
We aim to find the bounds of $\mathbb{P}_{\theta,id}(\widehat{\pi}_0 \ne id)$ given the dispersion parameter $\theta$.

For $\pi, \pi' \in \mathfrak{S}_n$, we say that $\pi$ is more likely than $\pi'$, denoted $\pi \succeq \pi'$, if 
$\sum_{i = 1}^N  \inv(\pi_i \circ \pi^{-1}) \le \sum_{i = 1}^N  \inv(\pi_i \circ \pi'^{-1})$.
For $1 \le j < k \le n$, let $(j,k)$ be the transposition of $j$ and $k$.
By the union bound, we get
\begin{equation*}
\mathbb{P}_{\theta,id}((1,2) \succeq id) \le \mathbb{P}_{\theta,id}(\widehat{\pi}_0 \ne id)  \leq  \sum_{\pi \in \mathfrak{S}_n} \mathbb{P}_{\theta,id}(\pi \succeq id).
\end{equation*}
{\bf Lower bound:} 
For any permutation $\pi \in \mathfrak{S}_n$, if $\pi(1) > \pi(2)$, then $\inv(\pi \circ (1,2)) = \inv(\pi) - 1$, 
and if $\pi(1) < \pi(2)$, then $\inv(\pi \circ (1,2)) = \inv(\pi) + 1$.
Thus, $\sum_{i = 1}^N \inv(\pi_i \circ (1,2))$ equals to
\begin{equation*}
 \sum_{i = 1}^N \inv(\pi_i) + \#\{i: \pi_i(1) < \pi_i(2) \} - \#\{i: \pi_i(1) >\pi_i(2) \}.
\end{equation*}
Consequently, $\mathbb{P}_{\theta,id}((1,2) \succeq id)$ is given by
\begin{align}
\label{eq:12}
& \quad ~ \mathbb{P}_{\theta,id}\Bigg( \#\{i: \pi_i(1) < \pi_i(2) \} \le \#\{i: \pi_i(1) >\pi_i(2) \}\Bigg) \notag\\
& = \mathbb{P}_{\theta,id}\left( \#\{i: \pi_i(1) >\pi_i(2)\}  \ge  \frac{N}{2} \right).
\end{align}
Note that $\mathbb{P}_{\theta,id}(\pi_i(1) > \pi_i(2)) = e^{-\theta} \mathbb{P}_{\theta,id}(\pi_i(1) < \pi_i(2))$ which implies that
\begin{equation}
\label{eq:Mallowshalf}
\mathbb{P}_{\theta,id}(\pi_i(1) > \pi_i(2)) = \frac{1}{1 + e^{\theta}}.
\end{equation}
Combining \eqref{eq:12} and \eqref{eq:Mallowshalf} yields
\begin{equation*}
\label{eq:MallowsLB}
\begin{aligned}
\mathbb{P}_{\theta,id}((1,2) \succeq id) &= \mathbb{P}\left(\Bin\left(N,  \frac{1}{1+e^{\theta}}\right) \ge \frac{N}{2} \right)   \\
                                              &  \sim \frac{1}{1- e^{-\theta}} \sqrt{\frac{2}{\pi N}} \left( \cosh\frac{\theta}{2}\right)^{-N},
\end{aligned}                                              
\end{equation*}
where $\Bin(N, p)$ is a binomial random variable with parameters $(N, p)$,
and the last estimate follows from the large deviation bound \cite{AG89} that for $p<a$, 
\begin{equation*}
\mathbb{P}(\Bin(N, p) > a N) \sim \frac{(1-p)\sqrt{a}}{(a-p)\sqrt{2 \pi (1-a)N}} e^{-N H(a, p)},
\end{equation*}
where $H(a,p) := a \log\left(\frac{a}{p} \right) + (1- a) \log\left(\frac{1- a}{1- p} \right)$ is the relative entropy between $\Bin(N, p)$ and $\Bin(N, a)$. 

{\bf Upper bound:} 
We need the following comparison result.
\begin{proposition}
\label{comparison}
For $j,k \in \{1, \ldots, n\}$ and $j < k$, we have
\begin{equation}
\label{eq:MallowsUB}
\mathbb{P}_{\theta,id}((j,k) \succeq id) \le \left(\cosh \frac{\theta}{2} \right)^{-N}.
\end{equation}
\end{proposition}

Now decompose each permutation $\pi \in \mathfrak{S}_n$ into a product of transpositions, say $\pi = \sigma_1 \circ \cdots \circ \sigma_k$. 
We have
\begin{align}
& \mathbb{P}_{\theta, id}(\pi \succeq id) \notag \\
& \qquad  \le  \mathbb{P}_{\theta, id} (\exists i \in [k]: \sigma_1 \circ \cdots \circ \sigma_i \succeq \sigma_1 \circ \cdots \circ \sigma_{i-1})     \notag \\
&\qquad \le \#_{trans}(\pi) \left(\cosh \frac{\theta}{2} \right)^{-N}, \label{eq:impinc}
\end{align}
where $\#_{trans}(\pi)$ is the number of transpositions in $\pi$.
For any permutation $\pi \in \mathfrak{S}_n$, $\#_{trans}(\pi) = n - \#_{cyc}(\pi)$ with $\#_{cyc}(\pi)$ the number of cycles in $\pi$. 
The upper bound \eqref{eq:upperbd} follows from \eqref{eq:impinc}, the union bound and the fact that 
$\sum_{\pi \in \mathfrak{S}_n} \#_{cyc}(\pi) = n! H_n$.

\begin{proof}[Proof of Proposition \ref{comparison}]
A similar argument as before shows that 
\begin{equation*}
\mathbb{P}((j, k) \succeq id) = \mathbb{P}((1,2) \succeq id) \quad \mbox{for } k - j =1.
\end{equation*}
Now let $\ell := k - j \in \{1, \ldots, n-1\}$.
It is not hard to see that for any permutation $\pi \in \mathfrak{S}_n$, the number of inversions $\inv(\pi \circ (j, k))$ can take $2 \ell$ values:
\begin{equation*}
\inv(\pi \circ (j, k)) \in \{\inv(\pi) \pm m: m = 1,3, \ldots, 2 \ell -1 ) \}.
\end{equation*}
For $m \in \{-2 \ell + 1, -2 \ell + 3, \ldots, 2 \ell -3, 2 \ell -1\}$, let
\begin{equation*}
p_m: = \mathbb{P}_{\theta,id}\Bigg(\inv(\pi \circ (j, k)) = \inv(\pi) + m\Bigg).
\end{equation*}
Observe that $\mathbb{P}_{\theta,id}((j,k) \succeq id)$ is given by
\begin{align}
\label{eq:estjk}
& \quad ~ \mathbb{P}_{\theta,id}\left( \sum_{i=1}^N \Bigg( \inv(\pi_i \circ (j,k)) - \inv(\pi_i)  \Bigg)\le 0   \right) \notag \\
& = \mathbb{P}\left(\sum_{i = 1}^N Z_i \le 0 \right),
\end{align}
where $Z_i$ are i.i.d. categorical random variables such that $Z_i = m$ with probability $p_m$ for $m \in \{-2 \ell + 1, -2 \ell + 3, \ldots, 2 \ell -3, 2 \ell -1\}$.
By Cramer's theorem \cite{DZ},
\begin{equation}
\label{eq:LDP}
 \mathbb{P}\left(\sum_{i = 1}^N Z_i \le 0 \right) \le \exp \left(-N \sup_{\lambda} \{- \log F(\lambda)\} \right),
\end{equation}
where 
\begin{equation*}
F(\lambda): = \sum_{m = -2 \ell + 1, \, m \, \mbox{\tiny odd}}^{2 \ell -1} p_m e^{\lambda m},
\end{equation*}
is the moment generating function of $Z_i$. 
Note that for $m > 0$, we have $p_{-m} = e^{\theta m} p_m$.
Therefore,
$\sum_{m = 1, \, m \, \mbox{\tiny odd}}^{2 \ell -1} p_m (1+e^{\theta m}) =  1$.
Moreover,
\begin{equation*}
\begin{aligned}
\sum_{m = 1, \, m \, \mbox{\tiny odd}}^{2 \ell -1} p_m (1+e^{\theta m}) &= \sum_{m = 1, \, m \, \mbox{\tiny odd}}^{2 \ell -1} p_m e^{\frac{\theta m}{2}}(e^{\frac{\theta m}{2}}+e^{-\frac{\theta m}{2}}) \\
& \ge \sum_{m = 1, \, m \, \mbox{\tiny odd}}^{2 \ell -1} p_m e^{\frac{\theta m}{2}}(e^{\frac{\theta}{2}}+e^{-\frac{\theta}{2}}),
\end{aligned}
\end{equation*}
which yields $\sum_{m = 1, \, m \, \mbox{\tiny odd}}^{2 \ell -1} p_m e^{\frac{\theta m}{2}} \le  \frac{1}{2} (\cosh \frac{\theta}{2})^{-1}$.
As a consequence,
\begin{equation}
\label{eq:dual}
\begin{aligned}
\sup_{\lambda} \{- \log F(\lambda)\} & \ge - \log F\left(\frac{\theta}{2} \right) \\
& = - \log  \sum_{m = 1, \, m \, \mbox{\tiny odd}}^{2 \ell -1} 2 p_m e^{\frac{\theta m}{2}} \\
& \ge - \log \left(\cosh \frac{\theta}{2} \right)^{-1}.
\end{aligned}
\end{equation}
By \eqref{eq:estjk}, \eqref{eq:LDP} and \eqref{eq:dual}, we get the estimate \eqref{eq:MallowsUB}.
\end{proof}

\section{Experimental Results}
In this section we apply Algorithm \ref{algo1} to provide experimental results on synthetic data and two real-world data: 
APA election data (large $N$, small $t_{max}$), and University's homepage search (small $N$, large $t_{max}$).

\subsection{Synthetic Data}
We generate $50$ sets of $N = 1000$ rankings from the IGM model \eqref{eq:IGM} with $\vec{\theta} = (1,0.9,0.8,0.7,0.6,0.5, 0, \ldots)$ and $\pi_0 = id$.
We restrict all observed rankings to the first $t_{\max} = 6$ components.
Table \ref{t1} displays the percentage of the model sizes selected by Algorithm \ref{algo1}.
\begin{table}[t]
\caption{Percentage of model sizes selected for $50$ simulated data sets.}
\label{t1}
\vskip 0.15in
\begin{center}
\begin{small}
\begin{sc}
\begin{tabular}{ | l | c | c | c | c | c | c |} 
 \hline
 Model size $t$ & $1$ & $2$ & $3$ & $4$ & $5$ & $6$\\
 \hline
 Percentage ($\%$)& $0$ & $65$ & $35$ & $0$ & $0$ & $0$   \\
 \hline 
 \end{tabular} 
\end{sc}
\end{small}
 \end{center}
 \vskip -0.1in
\end{table}

Now we fit the IGM model by MLE with the preselected model size $t$. 
The estimate $\widehat{\theta}_1$ lies in the range $[0.94, 1.08]$ with mean $1.0$ and standard deviation $0.03$, 
and $\widehat{\theta}_2$ in the range $[0.84, 0.95]$ with mean $0.9$ and standard deviation $0.02$.
Moreover, the estimated central rankings restricted to the top $6$ ranks are always $(1|2|3|4|5|6)$.

We also generate $50$ sets of $N = 1000$ rankings from the IGM model restricted to the first $t_{\max} = 10, 20$ and $40$ components.
Table \ref{t33} shows the accuracy of estimated ranking and average training time by the IGM model of model size $t = 1$, $t = 10$ and Algorithm \ref{algo1}. 
For rankings of small length, the IGM model of small size gives good accuracy and uses less training time. 
Algorithm \ref{algo1} is more appealing for rankings of large length.
It has better accuracy than the IGM model of small size, and is less time-consuming than the IGM model of large size. 

\begin{table}[t]
\caption{Accuracy of estimated rank $\&$ average training time for $50$ simulated data with $t_{max} = 10$ (resp. $t_{max} = 20$, $t_{max} = 40$) and $\vec{\theta} = (1, 0.975, \ldots, 0.775, 0, \ldots)$ (resp. $\vec{\theta} = (1, 0.975, \ldots, 0.525, 0 , \ldots)$, $\vec{\theta} = (1, 0.975, \ldots, 0.025, 0 , \ldots)$)
by the IGM model of model size $t = 1$, $t = 10$ and Algorithm \ref{algo1}.
}
\label{t33}
\vskip 0.15in
\begin{center}
\begin{small}
\begin{sc}
\begin{tabular}{ | l | c | c | c |} 
 \hline
 $t_{max} = 10$  & IGM($t = 1$) & IGM($t = 10$) & Alg \ref{algo1} \\
 \hline
 Acc. est. rank   & \pmb{$100 \%$} & \pmb{$100 \%$} & \pmb{$100 \%$}  \\
 \hline
 Ave. time & \pmb{$1.56$ s} & $14.45$ s & $2.80$ s    \\
 \hline 
 \hline
 $t_{max} = 20$  & IGM($t = 1$) & IGM($t = 10$) & Alg \ref{algo1} \\
 \hline
 Acc. est. rank   & $94 \%$ & \pmb{$100 \%$} & \pmb{$100 \%$}  \\
 \hline
 Ave. time & \pmb{$5.73$ s} & $54.45$ s & $24.42$ s    \\
 \hline 
 \hline
 $t_{max} = 40$  & IGM($t = 1$) & IGM($t = 10$) & Alg \ref{algo1} \\
 \hline
 Acc. est. rank   & $82 \%$ & \pmb{$100 \%$} & \pmb{$100 \%$}  \\
 \hline
 Ave. time & \pmb{$70.26$ s} & $684.65$ s & $391.20$ s    \\
 \hline 
 \end{tabular} 
\end{sc}
\end{small}
 \end{center}
 \vskip -0.1in
\end{table}

\subsection{APA Data}
We consider the problem of ranking with a small number of items and large sample size. 
The data consists of $N = 15449$ rankings over $t_{max} = 5$ candidates during the American Psychological Association's presidential election in $1980$. 
Among these rankings, there are only $5738$ complete rankings and the number of distinct rankings is $n = 207$.
See \cite{CCC, Diaconis89} for further background.

Table \ref{t3} displays the estimate of $\theta$ for the single parameter IGM with different model sizes. 
Applying Algorithm $1$ yields the selection of $t = 2$. 
Then by fitting the IGM model with $t = 2$, we get $\widehat{\theta}_1 = 0.46$, $\widehat{\theta}_2 = 0.54$ 
and $\widehat{\pi}_0 = (3|1|5|4|2)$.
By the second order analysis, \cite{Diaconis89} argued that there is a strong effect of choosing candidates $\{1,3\}$ and $\{4,5\}$, with candidate $\{2\}$ in the middle.
Our result suggests that the pair of candidates $\{1,3\}$ be in a more favorable position.

\begin{table}[t]
\caption{Estimates of $\theta$ for the single parameter IGM model with $t \in [1,5]$.}
\label{t3}
\vskip 0.15in
\begin{center}
\begin{small}
\begin{sc}
\begin{tabular}{ |c|c|c|c|c|c| } 
\hline
Model size $t$ & $1$  & $2$ & $3$ & $4$ & $5$  \\
\hline
Estimated $\theta$ & $0.47$  & $0.50$ &  $0.54$ & $0.62$ &  $0.72$ \\
 \hline 
\end{tabular} 
\end{sc}
\end{small}
 \end{center}
 \vskip -0.1in
\end{table}

\subsection{University's Homepage Search}
We consider the problem of learning a domain-specific search engine 
with the data collected by \cite{CSS}.
The data consists of $157$ universities, the {\em queries}, and $21$ search engines, the {\em experts}. 
Each expert search engine outputs a ranked list of up to $t_{max} = 30$ URLs when queried with the university's name.
The target output is the university's homepage.
There are $10$ universities without data, and some search engines return empty list.
So there are $147$ ranking problems with sample size $N \le 21$, and each sample has length ranging from $1$ to $30$.

For each query, we fit the IGM model by MLE with Algorithm \ref{algo1} to calculate the rank of the university's homepage under the estimated central ranking $\widehat{\pi}_0$.
This rank measures the correctness of the model.
If the target homepage is not among the URLs returned by the search engines, we put it to the end of the list.
The central ranking is estimated by the SORTR heuristic \cite{MB10}.
Table \ref{t4} is an extract of the estimated rank of the target homepages.

\begin{table}[t]
\caption{A list of $10$ universities and their estimated ranks.}
\label{t4}
\vskip 0.15in
\begin{center}
\begin{small}
\begin{sc}
\begin{tabular}{ | l | c| } 
 \hline
University's name &  Estimated rank  \\ 
\hline
Oregon Health Sciences Univ.  & $2$ \\ 
Ouachita Baptist Univ.  & $1$ \\ 
Our Lady of the Lake Univ. & $6$  \\
Pacific Baptist Univ. & $2$ \\
Pacific Christian College & $8$ \\
Pacific Lutheran Univ. & $1$ \\
Worcester State College & $1$ \\
Yeshiva Univ. & $19$ \\
York College of Pennsylvania & $6$ \\
Young Harris College & $1$ \\
 \hline 
\end{tabular} 
\end{sc}
\end{small}
 \end{center}
 \vskip -0.1in
\end{table}

Our training model is slightly different from that in \cite{MB10},
since they used the parameterization $\vec{\theta}= (\theta_1, \ldots, \theta_{t-1}, \theta_t, \ldots. \theta_t, 0 , \ldots)$ for the top-$t$ IGM model.
Figure \ref{fig:single} provides the estimates of $\theta$ for the single parameter IGM model with different model sizes. 
By Algorithm \ref{algo1}, $68 \%$ select $t = 5$, $20 \%$ select $t = 1$, and $11 \%$ select $t = 2$ over all $147$ queries.
We also computed the rank of the homepage for each query and each model size $t$.
Table \ref{t5} summarizes the mean and the median of the target rank for these models.

\begin{figure}[ht]
\vskip 0.2in
\begin{center}
\centerline{\includegraphics[width=1 \columnwidth]{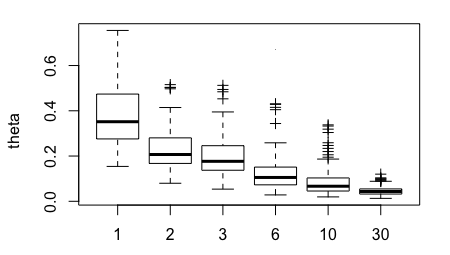}}
\caption{Estimates of $\theta$ for the single parameter IGM model with $t \in \{1,2,3,6,10,30\}$.}
\label{fig:single}
\end{center}
\vskip -0.2in
\end{figure}

\begin{table}[t]
\caption{Mean and median rank of the target homepage under the IGM models.}
\label{t5}
\vskip 0.15in
\begin{center}
\begin{small}
\begin{sc}
\begin{tabular}{ | l | c |  c  | c | c | c | c || c |} 
\hline
Size $t$ & $1$ & $2$ & $3$ & $6$ & $10$ & $30$ & Alg $1$ \\ 
\hline
Mean R.       & $3.7$ & $4.2$ & $4.9$ & $8.0$ & $13.2$ & $18.9$ & $6.8$ \\ 
\hline
Med. R.   & $2$  & $2$ & $2$ & $3$ & $5$ & $6$ & $3$ \\
\hline
\end{tabular} 
\end{sc}
\end{small}
 \end{center}
 \vskip -0.1in
\end{table}

\cite{MB10} got a mean rank around $15$ and a median rank around $10$.
Compared to their results, our experiments give better target rank for small model sizes. 
This is reasonable because for each query there are only $N \le 21$ samples and large model sizes may cause overfitting.
Small median ranks in Table \ref{t5} also supports the validity of the IGM model.

\section{Conclusion}
In this paper we study various Mallows ranking models. 
The parameters of interest are the dispersion $\theta$ (or $\vec{\theta}$), and the central ranking $\pi_0$. 
Though there have been efficient algorithms to estimate the MLE $(\widehat{\theta}, \widehat{\pi}_0)$, not much is known about the statistical properties of $(\widehat{\theta}, \widehat{\pi}_0)$ except the consistency of $\widehat{\theta}$.
Aiming to fill this gap, we prove the biasedness and convergence rate of the MLE of the Mallows' $\phi$ model and the single parameter IGM model.

To compare a large number of items, an infinite ranking model is often
useful. A natural infinite generalization of the finite Mallows model appeared in several contexts \cite{GO09, MB10}.
But neither of them are aware of the other.
In this work we make a clear connection between these infinite rankings with further analysis. 
Relying on a renewal property of the single parameter IGM model, we propose an algorithm to choose the model size `$t$' automatically.
The `$t$' selection algorithm is tested over synthetic and real data, and shows good performance.

\nocite{langley00}

\bibliography{unique}
\bibliographystyle{icml2019}
\end{document}